\documentclass[reqno]{amsart}
\usepackage{amssymb, amsthm, amsmath, xypic}
\usepackage[pagewise]{lineno}%\linenumbers
\usepackage{color}

\calclayout

\theoremstyle{definition}
\newtheorem{theorem}{Theorem}[section]
\newtheorem{definition}[theorem]{Definition}
\newtheorem{proposition}[theorem]{Proposition}
\newtheorem{lemma}[theorem]{Lemma}
\newtheorem{corollary}[theorem]{Corollary}
\newtheorem{example}[theorem]{Example}

\renewcommand{\Im}{\mathrm{Im}}

\DeclareMathOperator{\Perm}{\mbox{Perm}}
\DeclareMathOperator{\Hol}{\mbox{Hol}}
\DeclareMathOperator{\Aut}{\mbox{Aut}}

\DeclareMathOperator{\Stab}{\mbox{Stab}}

\DeclareMathOperator{\Map}{\mbox{Map}}

\newcommand{\calL}{\mathcal{L}}

\newcommand{\calR}{\mathcal{R}}

\newcommand{\id}{\mathrm{id}}

\begin{document}

\title[Skew bracoids containing a skew brace]{Skew bracoids containing a skew brace}

\author{Ilaria Colazzo}
\address{University of Exeter \\ Department of Mathematics and Statistics \\ Exeter \\ EX4 4QF \\ UK}
\curraddr{University of Leeds \\ School of Mathematics \\ Leeds \\ LS2 9JT \\ UK}
\email{I.Colazzo@leeds.ac.uk}

\author{Alan Koch}
\address{Department of Mathematics, Agnes Scott College, 141 E. College Ave., Decatur, GA 30030 USA}
\email{AKoch@agnesscott.edu}

\author{Isabel Martin-Lyons}
\address{School of Computer Science and Mathematics \\ Keele University \\ Staffordshire \\ ST5 5BG \\ UK}
\email{I.D.Martin-Lyons@keele.ac.uk}

\author{Paul J. Truman}
\address{School of Computer Science and Mathematics \\ Keele University \\ Staffordshire \\ ST5 5BG \\ UK}
\email{P.J.Truman@keele.ac.uk}



\thanks{For the purpose of Open Access, the authors have applied a CC BY public copyright licence to any Author Accepted Manuscript (AAM) version arising from this submission.} 

\begin{abstract}
Skew bracoids have been shown to have applications in Hopf-Galois theory. We show that a certain family of skew bracoids correspond bijectively with left cancellative semibraces. A consequence of this correspondence is that skew bracoids in this family can be used to obtain and study solutions of the set-theoretic Yang--Baxter equation; we study this process and the resulting solutions. We give numerous examples of skew bracoids satisfying our hypothesis, drawing upon a variety of constructions in the literature.
\end{abstract}

\maketitle

\section{Introduction} \label{sec_introduction}

Skew braces, introduced by Guarnieri and Vendramin in \cite{GV17}, provide an algebraic framework for studying bijective nondegenerate solutions of the set-theoretic Yang--Baxter equation (see \cite{ESS99} or Section \ref{sec_solutions} for more details on the set-theoretic YBE). A skew brace is a triple $ (G,\star,\cdot) $ in which $ (G,\star) $ and $ (G,\cdot) $ are groups and the compatibility relation
\begin{equation} \label{eqn_skew_brace_relation}
x \cdot (y \star z) = (x \cdot y) \star x^{-\star} \star  (x \cdot z) 
\end{equation}
holds for all $ x,y,z \in G $ (here $ x^{-\star} $ denotes the inverse of $ x $ with respect to the binary operation $ \star $). Thanks to the importance of the Yang--Baxter equation in theoretical physics and statistical mechanics \cite{Jim89}, \cite{Jim94}, skew braces have been intensively studied, revealing connections with numerous other topics including braids, knots, regular subgroups of permutation groups, and Hopf-Galois theory. This has led to a fruitful network of ideas, which has motivated the development of numerous generalizations and variants of skew braces. For example, in \cite{CCS17} left \textit{semibraces} are introduced; these are triples $ (G,+,\cdot) $ in which $ (G,\cdot) $ is a group, $ (G,+) $ is a left cancellative semigroup, and the binary operations are connected via a relation similar to \eqref{eqn_skew_brace_relation} (see Section \ref{sec_semibraces} for more details). Left semibraces yield \textit{left nondegenerate} solutions of the set-theoretic Yang--Baxter equation. In \cite{MLT24} two of the authors introduce \textit{skew bracoids}; these consist of two groups $ (G,\cdot) $ and $ (N,\star) $, together with a transitive action of the former on the latter that interacts with the binary operation on $ N $ in a manner analogous to \eqref{eqn_skew_brace_relation}. Skew bracoids can be used to broaden the connection between skew braces and Hopf-Galois theory \cite[Section 5]{MLT24}. 

In this paper we show that a certain large family of skew bracoids correspond bijectively with semibraces (Theorem \ref{thm_bracoid_semibrace}). This correspondence enriches the study of both objects: it shows that many skew bracoids yield solutions of the set-theoretic Yang--Baxter equation, it connects semibraces with Hopf-Galois theory, and it shows that results concerning the classification or description of either object can be reinterpreted in terms of the other. 

In Section \ref{sec_skew_bracoids} we recall some fundamental definitions and results concerning skew bracoids, specify the family of skew bracoids we will study (Definition \ref{def_contains_skew_brace}), establish a number of equivalent characterizations (Proposition \ref{prop_characterisation}), and use these to give a variety of examples. We also exhibit a family of skew bracoids that do not satisfy our hypothesis. In Section \ref{sec_semibraces} we prove several technical lemmas concerning skew bracoids satisfying Definition \ref{def_contains_skew_brace}, and use these to establish our main result (Theorem \ref{thm_bracoid_semibrace}). 

As noted above, an interesting consequence of the correspondence established in Theorem \ref{thm_bracoid_semibrace} is that it allows us to obtain a solution of the set-theoretic Yang--Baxter equation from certain skew bracoids by passing through the corresponding semibrace. In Section \ref{sec_solutions} we show how this solution may be obtained directly from the skew bracoid, and study properties of this construction. 

\subsection*{Acknowledgements}
Ilaria Colazzo was partially supported by EPSRC (Engineering and Physical Sciences Research Council) project reference EP/V005995/1. Paul J. Truman was supported by EPSRC  project reference EP/W012154/1. The authors gratefully acknowledge the support of the London Mathematical Society, joint research group 32312, which enabled parts of this collaboration to take place.

We are grateful to the anonymous referee, whose comments and suggestions led to several improvements in the exposition. 

\section{Skew bracoids containing a brace} \label{sec_skew_bracoids}

A (left) \textit{skew bracoid} is a quintuple $ (G,\cdot,N,\star,\odot) $ in which $ (G,\cdot) $ and $ (N,\star) $ are groups and $ \odot $ is a transitive (left) action of $ (G,\cdot) $ on $ N $ such that the equation
\begin{equation} \label{eqn_skew_bracoid_relation}
x \odot (\eta \star \mu) = (x \odot \eta) \star (x \odot e_{N})^{-1} \star (x \odot \mu) 
\end{equation}
holds for all $ x \in G $ and $ \eta, \mu \in N $ (see \cite{MLT24}). We call $ (G,\cdot) $ the \textit{multiplicative group}, and $ (N,\star) $ the \textit{additive group}, of the skew bracoid $ (G,\cdot,N,\star,\odot) $. We suppress the notation for the binary operations in $ (G,\cdot) $ and $ (N,\star) $ whenever possible, and often suppress the symbol $ \odot $ when specifying skew bracoids, \textit{viz} $ (G,N) $. 

For brevity, we shall refer to skew braces simply as \textit{braces} and skew bracoids simply as \textit{bracoids}. 

In this section we specify and characterize a large family of bracoids; in subsquent sections we will connect bracoids in this family with semibraces and with set-theoretic solutions of the Yang--Baxter equation. 

We say that a bracoid $ (G,N) $ in which the subgroup $ S = \Stab(e_{N}) $ of $ G $ is trivial is \textit{essentially a brace}, since in this case the operation on one group can be transported to a new operation on the other, which then becomes a skew brace (see \cite[Example 2.2]{MLT24}). We shall study bracoids satisfying the following variant of this condition.

\begin{definition} \label{def_contains_skew_brace}
We shall say that a bracoid $ (G,N) $ \textit{contains a brace} if there is a subgroup $ H $ of $ G $ such that $ (H,N) $ is essentially a brace. 
\end{definition} 

To justify our choice of terminology we note that if $ (G,N) $ is a bracoid containing a brace $ (H,N) $ then we may define a second group operation on $ H $ via the rule
\begin{equation} \label{eqn_essentially_brace}
(x \star_{H} y) \odot e_{N} = (x \odot e_{N}) \star (y \odot e_{N})
\end{equation}
such that $ (H,\star_{H}) \cong (N,\star) $ and $ (H,\star_{H},\cdot) $ is a brace. Similarly, we may transport the transitive action of $ G $ on $ N $ to a transitive action of $ G $ on $ H $ by setting $ x \odot_{H} h $ to be the unique element of $ H $ such that
\begin{equation} \label{eqn_odot_H}
(x \odot_{H} h) \odot e_{N} = (xh) \odot e_{N};
\end{equation}
then $ (G,\cdot,H,\star_{H},\odot_{H}) $ is a bracoid isomorphic to $ (G,N) $ (see \cite[Section 4]{MLT24}), and it is certainly natural to say that this bracoid contains the brace $ (H,\star_{H},\cdot) $. Thus when considering bracoids containing a brace we will usually, and without loss of generality, assume $ N \subseteq G $ and denote this subgroup by $ H $.

Our first result gives a variety of characterizations of bracoids containing a brace. We summarise two constructions from the literature, which will be employed in the proof. First: a pair of groups $ H,S $ is said to be \textit{matched} if there is a left action of $ S $ on $ H $ and a right action of $ H $ on $ S $ such that  
\[ \,^{s}(h_{1}h_{2}) = \,^{s}h_{1} \,^{s^{h_{1}}}h_{2} \mbox{ and } 
(s_{1}s_{2})^{h} = s_{1}^{\,^{s_{2}}h} s_{2}^{h}; \]
in this case the Cartesian product $ H \times S $, together with the binary operation
\[ (h,s)(h',s') = (h \,^{s}h', s^{h'} s') \]
is a group, denoted by $ H \bowtie S $. Second: the \textit{holomorph} of a group $ H $ is the group $ \Hol(H) = H \rtimes \Aut(H) $, which acts on $ H $ via the formula $ (h,\alpha)(k) = h \alpha(k) $ for all $ h,k \in H $ and $ \alpha \in \Aut(H) $. If $ G $ is a transitive subgroup of $ \Hol(H) $ then $ (G,H) $ is a bracoid \cite[Theorem 2.8]{MLT24}. 

\begin{proposition} \label{prop_characterisation}
The following data are equivalent: 
\begin{enumerate}
\item a bracoid $ (G,H) $ containing a brace; \label{item_characterisation_1}
\item a bracoid $ (G,H) $ in which $ G=HS $ exactly, where $ S = \Stab_{G}(e) $; \label{item_characterisation_2}
\item a matched pair of groups $ H,S $ and a further binary operation $ \star $ on $ H $ such that $ (H,\star,\cdot) $ is a brace and $ S $ acts on $ (H,\star) $ by automorphisms; \label{item_characterisation_3}
\item a group $ (H,\star) $ and a transitive subgroup of $ \Hol(H,\star) $ containing a regular subgroup. \label{item_characterisation_4}
\end{enumerate}
\end{proposition}
\begin{proof}
First suppose that we are given \eqref{item_characterisation_1}, and let $ S = \Stab_{G}(e) $. Since $ (H,\cdot) $ acts regularly on $ H $, the stabilizer of $ e $ in $ H $ is trivial, so $ H \cap S $ is trivial. Given $ x \in G $ there exists a unique element $ h \in H $ such that $ x \odot e = h \odot e $, so $ s = h^{-1}x \in S $, and so $ x = hs $. Hence $ H $ is a complement to $ S $ in $ G $, and so we obtain \eqref{item_characterisation_2}. 
\\ \\
Next suppose that we are given \eqref{item_characterisation_2}. Then $ H,S $ are a matched pair of groups, and there is a binary operation $ \star $ on $ H $ such that $ (H,\star) $ is a group and 
\[ x \odot (h_{1} \star h_{2}) = (x \odot h_{1}) \star (x \odot e)^{-\star} \star (x \odot h_{2}) \mbox{ for all } x \in G \mbox{ and } h_{1},h_{2} \in H. \]
In particular, for all $ s \in S $ we have
\[ s \odot (h_{1} \star h_{2}) = (s \odot h_{1}) \star (s \odot h_{2}) \mbox{ for all } h_{1}, h_{2} \in H, \]
and so $ S $ acts on $ H $ by automorphisms. Hence we obtain \eqref{item_characterisation_3}. 
\\ \\
Next suppose that we are given \eqref{item_characterisation_3}. To ease notation we denote the left regular representation of $ \cdot $ by $ \lambda_{\bullet} $ rather than $ \lambda_{\cdot} $. Since $ (H,\star,\cdot) $ is a brace we have $ \lambda_{\bullet}(H) \subseteq \Hol(N,\star) $ \cite[Theorem 4.2.]{GV17}. Since in addition $ S $ acts on $ (H,\star) $ by automorphisms  the image of the map $ \theta : H \bowtie S \rightarrow \Perm(H) $ defined by $ \theta(h,s)[k] = h \,^{s}(k) $ for all $ h,k \in H $ and $ s \in S $ is also contained in $ \Hol(H,\star) $. Now note that
\[ \theta((h,s)(h',s'))[k] = \theta(h \,^{s}h', (s^{h'})s')[k] = h \,^{s}h' \,^{(s^{h'}s')}[k], \]
which agrees with
\[ \theta(h,s)\theta(h',s')[k] = \theta(h,s) \left( h' \,^{s'}[k] \right) = h \,^{s}\left( h' \,^{s'}[k] \right) \]
since $ H,S $ are a matched pair of groups. Hence $ \theta $ is a homomorphism and $ \Im(\theta) \leq \Hol(H,\star) $. Since $ \theta(h,e)[e] = h $ for all $ h \in H $, the image of $ H \bowtie \{ e \} $ is a regular subgroup of $ \Im(\theta) $, and so we obtain \eqref{item_characterisation_4}. 
\\ \\
Finally, suppose that we are given \eqref{item_characterisation_4}, and let $ G $ be a transitive subgroup of $ \Hol(H,\star) $ containing a regular subgroup. Then $ (G,H,\odot) $, where $ \odot $ denotes the natural action of $ G $ on $ H $, is a bracoid (see \cite[Theorem 2.8]{MLT24}). Since $ G $ contains a regular subgroup, this bracoid contains a brace, and so (relabelling if necessary) we obtain \eqref{item_characterisation_1}.
\end{proof}

In the remainder of this section we use Proposition \ref{prop_characterisation} to construct examples of bracoids containing a brace. 

A large family of bracoids arises from a certain quotienting procedure on braces. Recall that a \textit{strong left ideal} of a brace $ (G,\star,\cdot) $ is a subset $ S $ of $ G $ such that $ (S,\star) \mathrel{\unlhd} (G,\star) $ and $ \gamma_{x}(s) \in S $ for all $ x \in G $ and $ s \in S $, where $ \gamma : G \rightarrow \Aut(G,\star) $ is the homomorphism defined by
\begin{equation} \label{eqn_gamma_function}
\gamma_{x}(y) = x^{-\star} \star (x \cdot y) \mbox{ for all } x,y \in G. 
\end{equation}  
(these conditions imply that $ (S,\cdot) \leq (G,\cdot) $).  If $ S $ is a strong left ideal of the brace $ (G,\star,\cdot) $ then $ (G,\cdot) $ acts transitively on the group $ (G/S, \star) $ by left translation of cosets and $ (G,\cdot,G/S,\star,\odot) $ is a bracoid in which $ \Stab_{G}(eS)=S $ \cite[Proposition 2.4]{MLT24}. If in addition $ (S,\cdot) $ has a complement $ (H,\cdot) $ in $ (G,\cdot) $ then Proposition \ref{prop_characterisation} implies that the bracoid $ (G,\cdot,G/S,\star,\odot) $ contains a brace.

\begin{example} \label{ex_strong_left_ideal}
Suppose that $ H, S $ are groups and that there is a homomorphism $ \alpha : S \rightarrow \Aut(H) $. Then the set $ G=H \times S $, together with the operations
\[ (h,s) \star (h',s') = (hh', ss'), \mbox{ and } (h,s)\cdot(h',s') = (h \alpha_{s}(h'), ss') \]
forms a brace (see \cite[Example 1.6]{SV18}). It is straightforward to check that the $ \gamma $-function of this brace is $ \gamma_{(h,s)}(h',s') = (\alpha_{s}(h'),s') $, and it follows quickly that $ S'=\{ e \} \times S $ is a strong left ideal. Since $ \{ e \} \times S $ has a complement $ H'=H \times \{ e \} $ in $ (H \times S,\cdot) $, the bracoid $(G,G/S)$ contains a brace.
\end{example}

We note that in this example $ H \times \{ e \} $ is a \textit{normal} complement to $ \{ e \} \times S $ in $ (G,\cdot) $. Bracoids $ (G,N) $ such that $ \Stab(e_{N}) $ has a normal complement in $ (G,\cdot) $ have important applications in Hopf-Galois theory;  in the framework of \cite[Section 5]{MLT24} they correspond to Hopf-Galois structures on so-called \textit{almost classically Galois} field extensions, which occupy a distinguished place in the theory: see \cite[Section 4]{GP87}, \cite{Ko98}, \cite{By07}, for example. It is therefore interesting to note that the corresponding bracoids fall under the scope of Proposition \ref{prop_characterisation}. 

In \cite{Koc21a} it is shown that given a group $ (G,\cdot) $ and an endomorphism $ \psi $ of $ (G,\cdot) $ with abelian image (an \textit{abelian map}) we may define a new binary operation $ \star = \star_{\psi} $ on $ G $ such that $ (G,\star,\cdot) $ is a brace. The strong left ideals in braces of this form are characterized in \cite{KTpre}, as follows: writing $ \phi(x)=x \psi(x)^{-1} $ for all $ x \in G $, a subgroup $ (S,\cdot) $ of $ (G,\cdot) $ is a strong left ideal if and only if $ [G,\phi(S)] \leq S $ (where the commutators are computed with respect to $ \cdot $). If this is the case then we may apply  \cite[Proposition 2.4]{MLT24} and form the bracoid $ (G,G/S) $, as described above. It is possible for bracoids of this form to contain a brace:

\begin{example}
Let $ p $ and $ q $ be distinct odd prime numbers, and let 
\[ G = \langle x,y,z \mid x^{pq} = y^{2} = z^{2} = e, \; yxy^{-1} = zxz^{-1} = x^{-1}, \; yz=zy \rangle \cong C_{pq} \rtimes (C_{2} \times C_{2}). \]
The map $ \psi: G \rightarrow G $ defined by $ \psi(x^{i}y^{j}z^{k}) = y^{j}z^{k} $ is an abelian map, and we have $ \phi(x^{i}y^{j}z^{k}) = x^{i} $. Now let $ S = \langle x^{q}, z \rangle $; then $ \phi(S) = \langle x^{q} \rangle $ and $ [G,\phi(S)] \subseteq \langle x^{q} \rangle \subseteq S $, so we may form the bracoid $ (G,G/S) $. Finally, note that the subgroup $ S = \langle x^{q}, z \rangle $ has a complement $ H = \langle x^{p}, y \rangle $ in $ G $ (neither of $ H,S $ is normal in $ G $). Therefore by Proposition \ref{prop_characterisation} the bracoid $ (G,G/S) $ contains a brace
\end{example}

We can also construct examples of bracoids containing a brace that do not arise as the quotient of a brace by a strong left ideal as in Example \ref{ex_strong_left_ideal}. 

\begin{example}
Let $ N $ be an elementary abelian group of order $ 8 $. By \cite[Theorem 3.6]{By24} the holomorph of $ N $ contains a transitive subgroup $ J $ isomorphic to the simple group $ \mathrm{GL}_{3}(\mathbb{F}_{2}) $ of order $ 168 $. The resulting bracoid $ (J,N) $ does not arise as the quotient of a brace by a strong left ideal: see \cite[Example 2.22]{MLT24}. However, writing $ S = \Stab(e_{N}) $ as usual we have $ |S|=21 $ by the Orbit-Stabilizer Theorem, so taking $ H $ to be a Sylow $ 2 $-subgroup of $ J $ we have an exact factorization $ J = HS $, and so by Proposition \ref{prop_characterisation} the bracoid $ (J,N) $ contains a brace
\end{example}

We can also use Proposition \ref{prop_characterisation} to construct examples of bracoids that do not contain a brace. We seek a group $ N $ and a transitive subgroup $ J $ of $ \Hol(N) $ such that $ S = \Stab(e_{N}) $ does not have a complement in $ J $. Using results of Darlington \cite[Section 3.1]{Dar24} we have

\begin{example} \label{ex_minimal_transitive}
Let $ p $ and $ q $ be prime numbers with $ p \equiv 1 \pmod{q^{2}} $, and let $ N $ be a cyclic group of order $ pq $, presented as
\[ N = \langle \sigma, \tau \mid \sigma^{p}=\tau^{q}=1, \; \sigma\tau = \tau\sigma \rangle. \]
Let $ \alpha $ be an automorphism of $ N $ that has order $ q^{2} $ and fixes $ \tau $ (for example, let $ \alpha $ be the $ (p-1)/q^{2} $ power of a generator of $ \Aut(\langle \sigma \rangle) $), and let 
\[ J = \langle (\sigma, \id), (\tau, \alpha) \rangle \leq \Hol(N). \]
Then $ J $ is a transitive subgroup of $ \Hol(N) $ of order $ pq^{2} $, and $ (J,N) $ is a bracoid. 

We claim that $ S = \Stab(e_{N}) $ does not have a complement in $ J $. First we determine $ S $. Since $ \alpha $ fixes $ \tau $ we have $ (\tau, \alpha)^{i} = (\tau^{i}, \alpha^{i}) $ for each $ i $; it follows that $ S = \langle (e_{N},\alpha^{q}) \rangle $ and that $ |S| = q $. If $ H $ is a complement to $ S $ in $ J $ then in particular there exist $ h \in H $ and $ s \in S $ such that $ hs = (\tau, \alpha) $; hence $ (\tau,\alpha^{i}) \in H $ for some $ i \equiv 1 \pmod{q} $. An element of this form has order $ q^{2} $, but $ H $ must have order $ pq $, so this is impossible. Hence $ S $ does not have a complement in $ J $. 
\end{example}

%More generally, given a bracoid $ (G,N,\odot) $ the image of the permutation representation $ \lambda_{\odot} : G \rightarrow \Perm(N) $ is contained in $ \Hol(N) $ \cite[Theorem 2.8]{MLT24}. If $ S = \Stab(e_{N}) $ has a complement $ H $ in $ G $ then the transitive subgroup $ \lambda_{\odot}(G) \leq \Hol(N) $ contains a regular subgroup (namely $ \lambda_{\odot}(H) $). If $ \lambda_{\odot} $ is injective (that is, if the bracoid is \textit{reduced} \cite[Definition 2.14]{MLT24}) then the converse is true. We expect the situation of Example \ref{ex_minimal_transitive} to be relatively rare. 

\section{Connections with semibraces} \label{sec_semibraces}

In \cite{CCS17} the notion of a left \textit{semibrace} is introduced, and it is shown that these structures yield left nondegenerate solutions of the set-theoretic Yang--Baxter equation. In this section we obtain a correspondence between left bracoids containing a brace and left semibraces. One consequence of this correspondence is a procedure for obtaining solutions from bracoids containing a skew brace; we study this in more detail in Section \ref{sec_solutions}.

We begin by recalling some definitions and properties concerning semibraces. In \cite{CCS17} a left semibrace is defined\footnote{In \cite{JVA19}, Jespers and Van Antwerpen give a broader definition of semibraces; the left semibraces defined in \cite{CCS17} are what they call left cancellative left semibraces.} to be a triple $ (G,+,\cdot) $ in which $ (G,\cdot) $ is a group, $ (G,+) $ is a left cancellative semigroup, and the equation 
\begin{equation} \label{eqn_semibrace_relation}
x \cdot (y+z) = x \cdot y + x\cdot(x^{-1} + z)
\end{equation}
holds for all $ x,y,z \in G $. As when working with bracoids, we will suppress the notation $ \cdot $ where possible.

If $ (G,+,\cdot) $ is a left semibrace then for $ x,y \in G $ we define 
\[ \calL_{x}(y) = x(x^{-1} + y); \]
each $ \calL_{x} $ is an automorphism of the semigroup $ (G,+) $, and the map $ \calL : (G,\cdot) \rightarrow \Aut(G,+) $ is a homomorphism \cite[Proposition 3]{CCS17}. We may rewrite Equation \eqref{eqn_semibrace_relation} as
\[ x (y+z) = xy + \calL_{x}(z). \]
The semigroup $ (G,+) $ admits a decomposition $ (G+e,+) \oplus (E,+) $ (where $e$ is the identity of the group $(G,\cdot)$) in which $ (G+e,+) $ is a group and $ E $ denotes the set of idempotents of $ (G,+) $. It follows quickly from the semibrace relation \eqref{eqn_semibrace_relation} and the left cancellative property that $ e \in E $, that $ x \in E $ if and only if $ x+e = e $, and that if $ x \in E $ then $ x+y = y $ for all $ y \in G $.

% calR function for Yang-baxter connection
%\[ \calR_{y}(x) = \calL_{x}(y)^{-1}xy = (x^{-1} + y)^{-1}y; \]
%the map $ s : G \times G \rightarrow G \times G $ defined by $ s(x,y) = (\calL_{x}(y), \calR_{y}(x)) $ for all $ x,y \in G $ is then a left nondegenerate solution on $ G $. 

Next we turn to bracoids containing a brace. In \cite[Definition 2.10]{MLT24} the $ \gamma $-function of a bracoid is defined, analogous to the $ \gamma $-function of a brace mentioned in \eqref{eqn_gamma_function}. In the case of a bracoid $ (G,H) $ containing a brace this is a homomorphism $ \gamma : G \rightarrow \Aut(H,\star) $ defined by
\[ \gamma_{x}(h) = (x \odot e)^{-\star} \star (x \odot h) \mbox{ for all } x \in G \mbox{ and } h \in H. \]
Using this we define a function $ \lambda : G \rightarrow \Map(G,H) $ by 
\begin{equation} \label{eqn_lambda_def}
\lambda_{x}(y) = \gamma_{x}(y \odot e) \mbox{ for all } x,y \in G  
\end{equation}
and a function $ \rho : G \rightarrow \Map(G,G) $ by
\begin{equation} \label{eqn_rho_def}
\rho_{y}(x) = \lambda_{x}(y)^{-1}xy \mbox{ for all } x,y \in G.
\end{equation}
Thus for each $ x \in G $ the function $ \lambda_{x} $ is closely related to $ \gamma_{x} $ but has domain equal to $ G $ rather than $ H $. We note an important consequence of \eqref{eqn_odot_H} is that if $ h \in H $ then $ h \odot e = h $ and so (for example) $ \lambda_{x}(h) = \gamma_{x}(h \odot e) = \gamma_{x}(h) $. We establish two technical lemmas concerning properties of the functions $ \lambda $ and $ \rho $. 

\begin{lemma} \label{lem_lambda_homomorphism}
For $ x,y \in G $ we have $ \lambda_{xy} = \lambda_{x} \lambda_{y} $. 
\end{lemma}
\begin{proof}
Let $ x,y,z \in G $. Then
\begin{align*}
\lambda_{x}\lambda_{y}(z) = \; & \lambda_{x} (\gamma_{y}(z \odot e))\\
= \;  & \gamma_{x} ( \gamma_{y}(z \odot e) \odot e )  \\
= \;  & \gamma_{x} \gamma_{y}(z \odot e) \tag{since $ \gamma_{y}(z \odot e) \in H $ } \\
= \;  & \gamma_{xy} (z \odot e) \tag{$ \gamma $ is a homomorphism of groups} \\
= \;  & \lambda_{xy}(z).
\end{align*}
\end{proof}

\begin{lemma} \label{lem_rho_antihomomorphism}
For $ x,y \in G $ we have $ \rho_{xy} = \rho_{y} \rho_{x} $.  
\end{lemma}
\begin{proof}
Let $ x,y,z \in G $. Then we have
\begin{equation} \label{eqn_rho_1}
\rho_{xy}(z) = \lambda_{z}(xy)^{-1} zxy,
\end{equation} 
whereas
\begin{eqnarray}
\rho_{y} \rho_{x}(z) & = & \rho_{y} \left( \lambda_{z}(x)^{-1} zx \right) \nonumber \\
& = & \lambda_{\lambda_{z}(x)^{-1}zx}(y)^{-1} \lambda_{z}(x)^{-1} zxy. \label{eqn_rho_2}
\end{eqnarray}
Using Lemma \ref{lem_lambda_homomorphism}, we see that \eqref{eqn_rho_1} and \eqref{eqn_rho_2} agree if and only if
\begin{equation} \label{eqn_rho_sufficient}
\lambda_{\lambda_{z}(x)^{-1}} \lambda_{zx}(y) = \lambda_{z}(x)^{-1} \lambda_{z}(xy).
\end{equation}
To simplify notation, let $ h = \lambda_{z}(x)^{-1} \in H $; then the left hand side of \eqref{eqn_rho_sufficient} is equal to
\begin{align}
\lambda_{h} (\lambda_{zx}(y) \odot e) = \; & \gamma_{h}(\lambda_{zx}(y)) \tag{since $ \lambda_{zx}(y) \in H $} \nonumber \\
= \; & (h \odot e)^{-\star} \star (h \odot \lambda_{zx}(y)) \nonumber \\
= \; & h \odot ((h^{-1} \odot e) \star \lambda_{zx}(y)) \tag{ by \eqref{eqn_skew_bracoid_relation}}  \nonumber \\
= \; & h \odot (h^{-1} \star \lambda_{zx}(y)), \label{eqn_rho_sufficient_LHS}
\end{align}
and the right hand side of \eqref{eqn_rho_sufficient} is equal to
\begin{align}
h \lambda_{z}(xy) = \; & h \odot \gamma_{z}(xy \odot e) \tag{ since $ h \in H $ } \nonumber \\
= \; & h \odot \gamma_{z}((x \odot e) \star \gamma_{x}(y \odot e))  \nonumber \\
= \; & h \odot (\gamma_{z}(x \odot e) \star \gamma_{zx}(y \odot e)) \nonumber \\
= \; & h \odot (h^{-1} \star \lambda_{zx}(y)) \label{eqn_rho_sufficient_RHS}
\end{align}
Since \eqref{eqn_rho_sufficient_RHS} and \eqref{eqn_rho_sufficient_LHS} agree, \eqref{eqn_rho_sufficient} holds, which completes the proof. 
\end{proof}

Finally, we record two corollaries of Lemma \ref{lem_rho_antihomomorphism}. 

\begin{corollary} \label{cor_rho_bijective}
For $ x \in G $, the function $ \rho_{x} $ is bijective.
\end{corollary}
\begin{proof}
First note that for all $ y \in G $ we have $ \lambda_{y}(e) = \gamma_{y}(e) = e $ since $ \gamma_{y} \in \Aut(H,\star) $; hence for all $ y \in G $ we have $ \rho_{e}(y) = \lambda_{y}(e)^{-1}ye = y $ and so $ \rho_{e} = \mathrm{id} $. Now by Lemma \ref{lem_rho_antihomomorphism}, for all $ x \in G $ we have 
\[ \rho_{x}\rho_{x^{-1}} = \rho_{x^{-1}}\rho_{x} = \rho_{e} = \mathrm{id}, \]
so $ \rho_{x} $ is bijective with inverse $ \rho_{x^{-1}} $. 
\end{proof}

\begin{corollary} \label{cor_lambda_useful}
For $ x,y,z \in G $ we have
\[ \lambda_{x}(yz) = \lambda_{x}(y) \lambda_{\rho_{y}(x)}(z). \]
\end{corollary}
\begin{proof}
By repeated application of \eqref{eqn_rho_def} we have
\begin{align}
\lambda_{x}(yz) \rho_{yz}(x) = \; & xyz \nonumber \\
= \; & \lambda_{x}(y) \rho_{y}(x) z \nonumber \\
= \; & \lambda_{x}(y) \lambda_{\rho_{y}(x)}(z) \rho_{z}(\rho_{y}(x)) \nonumber \\
= \; &\lambda_{x}(y) \lambda_{\rho_{y}(x)}(z) \rho_{yz}(x)  \tag{ by Lemma \ref{lem_rho_antihomomorphism}}; \nonumber
\end{align}
the result follows immediately. 
\end{proof}

With these results to hand, we can state and prove our main result.

\begin{theorem} \label{thm_bracoid_semibrace}
Let $ G = (G,\cdot) $ be a group with an exact factorization $ HS $. There is a bijection between
\begin{enumerate}
\item binary operations $ \star $ on $ H $ and transitive actions $ \odot $ of $ (G,\cdot) $ on $ H $ such that $ (G,\cdot,H,\star,\odot) $ is a left bracoid containing a brace and with $ \Stab_{G}(e)=S $;
\item binary operations $ + $ on $ G $ such that $ (G,+,\cdot) $ is a left semibrace in which $ G+e = H $ and $ E=S $.
\end{enumerate}
\end{theorem}
\begin{proof}
First suppose that $ (G,\cdot,H,\star,\odot) $ is a left bracoid containing a brace and with $ \Stab(e)=S $. For $ x,y \in G $ define 
\[ x + y = y \lambda_{y^{-1}}(x). \]
We claim that $ (G,+,\cdot) $ is a left semibrace. First we show that $ + $ is associative on $ G $. For $ x,y,z \in G $ we have
\begin{align*}
x + (y + z) = \; & x + z\lambda_{z^{-1}}(y) \\
 = \; & z\lambda_{z^{-1}}(y)\lambda_{(z\lambda_{z^{-1}}(y))^{-1}}(x) \\
 = \; & z\lambda_{z^{-1}}(y)\lambda_{\rho_{y}(z^{-1})y^{-1}}(x) \tag{by \eqref{eqn_rho_def}},
\end{align*}
and
\begin{align*}
(x+y) + z  = \; & y\lambda_{y^{-1}}(x) + z \\
 = \; & z\lambda_{z^{-1}}(y\lambda_{y^{-1}}(x)) \\
 = \; & z\lambda_{z^{-1}}(y)\lambda_{\rho_{y}(z^{-1})}(\lambda_{y^{-1}}(x)) \tag{by Corollary \ref{cor_lambda_useful} } \\
 = \; & z\lambda_{z^{-1}}(y)\lambda_{\rho_{y}(z^{-1})y^{-1}}(x) \tag{ by Lemma \ref{lem_lambda_homomorphism}}.
\end{align*}
Thus $ + $ is associative, and so $ (G,+) $ is a semigroup. 

Next we show that $ (G,+) $ is left cancellative. For $ x,y,z \in G $ we have
\begin{align*}
& x + y = x + z \\
\Rightarrow \; & y\lambda_{y^{-1}}(x) = z \lambda_{z^{-1}}(x) \\ 
\Rightarrow \; & x\rho_{x}(y^{-1})^{-1} = x\rho_{x}(z^{-1})^{-1} \\
\Rightarrow \; & \rho_{x}(y^{-1}) = \rho_{x}(z^{-1}) \\
\Rightarrow \; & y = z \tag{ by Corollary \ref{cor_rho_bijective}}.
\end{align*}
Thus $ (G,+) $ is left cancellative. 

Finally we show that the left semibrace relation is satisfied. For $ x,y,z \in G $ we have 
\begin{align*}
xy + x(x^{-1}+z) = \; & xy + xz\lambda_{z^{-1}}(x^{-1}) \\
 = \; & xz\lambda_{z^{-1}}(x^{-1}) \lambda_{(xz\lambda_{z^{-1}}(x^{-1}))^{-1}}(xy) \\
 = \; & xz\lambda_{z^{-1}}(x^{-1})\lambda_{\rho_{x^{-1}}(z^{-1})}(xy) \\
 = \; & xz\lambda_{z^{-1}}(x^{-1}xy) \tag{by Corollary \ref{cor_lambda_useful}} \\
 = \; & xz\lambda_{z^{-1}}(y) \\
 = \; & x(y+z). 
\end{align*}
Thus $ (G,+,\cdot) $ is indeed a left semibrace. 

For the remaining claims, we note that for $ x \in G $ we have $ x+e = e\lambda_{e^{-1}}(x) = \lambda_{e}(x) $; thus $ G+e = \lambda_{e}(G) = H $. Finally, we have
\begin{align*}
& x + x = x \\
\Leftrightarrow \; & x\lambda_{x^{-1}}(x) = x \\
\Leftrightarrow \; & \lambda_{x^{-1}}(x) = e \\
\Leftrightarrow \; & \lambda_{x^{-1}}(x) \odot e = e \tag{note $ \lambda_{x^{-1}}(x) \in H $} \\
\Leftrightarrow \; & \gamma_{x^{-1}}(x \odot e) = e \tag{relationship between $ \lambda $ and $ \gamma $} \\
\Leftrightarrow \; & x \odot e = e \tag{$\gamma_{x} \in \Aut(H,\star) $} \\
\Leftrightarrow \; & x \in S.  
\end{align*}
Thus $ E=S $. 

Conversely, suppose that $ + $ is a binary operation on $ G $ such that $ (G,+,\cdot) $ is a left semibrace in which $ G+e = H $ and $ E=S $. For $ h,k \in H $ define
\[ h \star k = k + h; \]
then $ (H,\star) $ is the opposite group to $ (H,+) $. In addition, for $ x \in G $ and $ h \in H $ define $ x \odot h = xh + e $. We claim that $ (G,\cdot,H,\star,\odot) $ is a bracoid containing a brace such that $ \Stab(e)=S $. It is straightforward to verify that $ \odot $ is a transitive left action of $ (G,\cdot) $ on $ H $; it remains to show that the bracoid relation \eqref{eqn_skew_bracoid_relation} is satisfied. By the definition of the binary operation $ \star $, it suffices to show that
\begin{equation} \label{eqn_semibrace_to_bracoid_sufficient}
x \odot (h + k) = (x \odot h) - (x \odot e) + (x \odot k) 
\end{equation}
for all $ x \in G $ and $ h,k \in H $. We shall use repeatedly the fact that since $ e \in E $ we have $ e+y=y $ for all $ y \in G $, along with the semibrace relation \eqref{eqn_semibrace_relation}. Beginning with the right hand side of \eqref{eqn_semibrace_to_bracoid_sufficient} we have:
\begin{eqnarray*}
(x \odot h) - (x \odot e) + (x \odot k)  & = & (x \odot h) - (x \odot e) + (xk + e) \\
& = & (x \odot h) - (x \odot e) + (x(e+k) + e) \\
& = & (x \odot h) - (x \odot e) + xe + \calL_{x}(k) + e \\
& = & (x \odot h) - (x \odot e) + xe + e + \calL_{x}(k) + e \\
& = & (x \odot h) - (x \odot e) + (x \odot e) + \calL_{x}(k) + e \\
& = & xh + \calL_{x}(k) + e \\
& = & x(h+k)+e \\
& = & x \odot (h+k). 
\end{eqnarray*}
Hence \eqref{eqn_semibrace_to_bracoid_sufficient} holds, and so $ (G,\cdot,H,\star,\odot) $ is indeed a left bracoid.

For the remaining claims: first we note that $ x \in S $ if and only if $ x + e = e $, which occurs if and only if $ x \in E $; second, for $ h \in H $ we have $ h \odot e = h + e = h $, so the action of $ H $ on itself is transitive, and so $ (H,+,\cdot) $ is a brace. 

The constructions described above are mutually inverse: if $ (G,\cdot,H,\star,\odot) $ is a bracoid as in the statement of the theorem and $ (G,+,\cdot) $ is the corresponding semibrace, then the bracoid corresponding to $ (G,+,\cdot) $ is $ (G,\cdot,H,\hat{\star}, \hat{\odot}) $, where for all $ h,k \in H $ and $ x \in G $ we have $ h\; \hat{\star} \; k = k + h $ and $ x \; \hat{\odot} \; h = xh + e $. Since $ k + h = h \star k $, we quickly see that $ \hat{\star} $ coincides with $ \star $. Since $ \hat{\odot} $ and $ \odot $ are both transitive actions of $ (G,\cdot) $ on $ H $, they coincide if and only if $ x \; \hat{\odot} \; e = x \odot e $ for all $ x \in G $. Writing $ x=hs $, with $ h \in H $ and $ s \in S $, we have: 
\[ x \;\hat{\odot}\; e = \; hs \;\hat{\odot}\; e = h + e = hs \odot e = x \odot e. \]
Thus the bracoid corresponding to $ (G,+,\cdot) $ is the original bracoid $ (G,\cdot,H,\star,\odot) $.

Conversely, if $ (G,+,\cdot) $ is a semibrace as in the statement of the theorem and $ (G,\cdot,H,\star,\odot) $ is the corresponding bracoid, then the semibrace corresponding to $ (G,\cdot,H,\star,\odot) $ is $ (G,\hat{+},\cdot) $, where $ x \;\hat{+}\; y = y\lambda_{y^{-1}}(x) $. We have
\begin{align*}
x \;\hat{+}\; y = \; & y\lambda_{y^{-1}}(x) \\
= \; & y\gamma_{y^{-1}}(x \odot e) \tag{by \eqref{eqn_lambda_def}} \\
 = \; & y((y^{-1}\odot e)^{-1} \star (y^{-1}\odot(x\odot e)) \tag{by the definition of $ \gamma_{y^{-1}} $}  \\
 = \; & y(y^{-1} \odot ((y \odot e)\star(x \odot e))) \tag{by \eqref{eqn_skew_bracoid_relation}} \\
 = \; & y(y^{-1} \odot ((y + e)\star(x + e))) \\
 = \; & y(y^{-1} \odot (x+e+y+e)) \\
 = \; & y(y^{-1}(x+y+e) + e) \\
 = \; & y y^{-1}(x+y+e) + y(y^{-1}+e) \tag{by \eqref{eqn_semibrace_relation}} \\
 = \; & x+y + y(y^{-1}+e) \\
 = \; & x+y (e + e) \tag{again by \eqref{eqn_semibrace_relation}} \\
 = \; & x+y. 
\end{align*}
Thus the semibrace corresponding to $ (G,\cdot,H,\star,\odot) $ is the original semibrace $ (G,+,\cdot) $. 

Hence we obtain the bijection given in the statement of the theorem.
\end{proof}

\section{Solutions of the set-theoretic Yang--Baxter equation} \label{sec_solutions}
A \textit{solution of the set-theoretic Yang--Baxter equation} on a nonempty set $ G $ (hereafter, simply a \textit{solution} on $ G $) is a map $ r : G \times G \rightarrow G \times G $ such that
\[ (r \times \id)(\id \times r)(r \times \id) = (\id \times r)(r \times \id)(\id \times r) \]
as functions on $ G \times G \times G $. We say that solution is \textit{bijective} if $ r $ is a bijective function, and \textit{involutive} if $ r^{2} = \id $. For $ x,y \in G $ we write
\[ r(x,y) = (\lambda_{x}(y), \rho_{y}(x)); \]
we say that a solution is \textit{left nondegenerate} if $ \lambda_{x} $ is bijective for each $ x \in G $, \textit{right nondegenerate} if $ \rho_{y} $ is bijective for each $ y \in G $, and \textit{nondegenerate} if it is both left and right nondegenerate. 

As stated in Section \ref{sec_introduction}, braces yield bijective nondegenerate solutions: if $ (G,\star,\cdot) $ is a brace and we set 
\begin{equation} \label{soln_skew_brace}
\lambda_{x}(y) = x^{-\star} \star (xy) \mbox{ and } \rho_{y}(x) = \lambda_{x}(y)^{-1}xy 
\end{equation}
then $ r(x,y)=(\lambda_{x}(y), \rho_{y}(x)) $ is such a solution \cite[Theorem 3.1]{GV17}. 

In \cite[Theorem 9]{CCS17} it is shown that if $ (G,+,\cdot) $ is a left semibrace and we set 
\begin{equation} \label{eqn_soln_semibrace}
\calL_{x}(y) = x(x^{-1} + y) \mbox{ and } \calR_{y}(x) = \calL_{x}(y)^{-1}xy
\end{equation}
then $ r(x,y)=(\calL_{x}(y), \calR_{y}(x)) $ is a left nondegenerate solution. The structure of a solution arising in this way is further explored in \cite{CCS20}, using the notion of the \textit{matched product} of solutions \cite[Definition 1 and Theorem 2]{CCS20}. Recall that we have a decomposition $ (G,+) = (H,+) \oplus (E,+) $ where $ H=G+e $ and $ E $ is the set of idempotents with respect to $ + $. The solution $ r(x,y) $ restricts to each of $ H $ and $ E $. In fact, $ (H,+,\cdot) $ is a left skew brace and $ (E,+,\cdot) $ is a trivial left semibrace, and the solution arising from the original semibrace $ (G,+,\cdot) $ is isomorphic to the matched product of the solution arising from $ (H,+,\cdot) $ and the solution arising from $ (E,+,\cdot) $ \cite[Theorem 9 and Theorem 10]{CCS20}. 

Applying the correspondence obtained in Theorem \ref{thm_bracoid_semibrace}, we can obtain a solution from a bracoid containing a brace by passing through the corresponding semibrace. We recall the functions $ \lambda, \rho $ defined in \eqref{eqn_lambda_def} and \eqref{eqn_rho_def}. 

\begin{proposition} \label{prop_bracoid_YBE_1}
Let $ (G,H) $ be a bracoid containing a brace. Then the function $ r : G \times G \rightarrow G \times G $ defined by
\[ r(x,y) = (\rho_{x^{-1}}(y^{-1})^{-1}, \lambda_{y^{-1}}(x^{-1})^{-1}) \]
is a left nondegenerate solution. 
\end{proposition}
\begin{proof}
Let $ (G,+,\cdot) $ be the semibrace corresponding to $ (G,H) $ via Theorem \ref{thm_bracoid_semibrace}. Then for $ x,y \in G $ we have $ x+y = y\lambda_{y^{-1}}(x) $. We rewrite the solution arising from the semibrace $ (G,+,\cdot) $ in terms of $ \lambda $ and $ \rho $. 
\begin{align*}
r(x,y)=(\calL_{x}(y), \calR_{y}(x)) = \; & (x(x^{-1} + y), \calL_{x}(y)^{-1}xy) \\
= \; & (xy\lambda_{y^{-1}}(x^{-1}),(xy\lambda_{y^{-1}}(x^{-1}))^{-1}xy) \\
= \; & (\rho_{x^{-1}}(y^{-1})^{-1},\lambda_{y^{-1}}(x^{-1})^{-1}).
\end{align*}
\end{proof}

The large number of inverses appearing in the formula for the solution in Proposition \ref{prop_bracoid_YBE_1} makes it rather unwieldy. By applying some standard techniques we can relate this left nondegenerate solution to a right nondegenerate solution, which more closely resembles the solution obtained from a skew brace in \eqref{soln_skew_brace}. 

\begin{proposition} \label{prop_bracoid_YBE_2}
Let $ (G,H) $ be a bracoid containing a brace. Then the function $ \tilde{r} : G \times G \rightarrow G \times G $ defined by
\[ \tilde{r}(x,y) = (\lambda_{x}(y), \rho_{y}(x)) \]
is a right nondegenerate solution. 
\end{proposition}
\begin{proof}
Define maps $ \iota, \tau : G \times G \rightarrow G \times G $ by
\[ \iota(x,y) = (x^{-1},y^{-1}) \mbox{ and } \tau(x,y)=(y,x). \]
It is clear that $ \iota, \tau $ are both self-inverse bijections. It is known (see \cite{ESS99}, for example) that conjugating a left nondegenerate solution by $ \iota $ yields another left nondegenerate solution, and that conjugating a left nondegenerate solution by $ \tau $ yields a right nondegenerate solution. Taking  $ r(x,y) $ as in Proposition \ref{prop_bracoid_YBE_1} we have:
\begin{align*}
\tau \iota r \iota \tau (x,y) = \; & \tau \iota r (y^{-1},x^{-1}) \\
= \; & \tau \iota (\rho_{y}(x)^{-1},\lambda_{x}(y)^{-1}) \\
= \; & (\lambda_{x}(y), \rho_{y}(x)).
\end{align*}
Hence $ \tilde{r}(x,y) $ is a right nondegenerate solution. 
\end{proof}

Since the solution obtained from the bracoid in Proposition \ref{prop_bracoid_YBE_1} is the same as the solution arising from corresponding semibrace, and the solution obtained in Proposition \ref{prop_bracoid_YBE_2} is closely related to it, they each restrict to the subgroups $ H $ and $ S $ of $ G $, and each is isomorphic to a suitable matched product of those restrictions.

\end{document}